\documentclass{amsart}

\usepackage{comment}
\begin{document}

\newtheorem{theorem}{Theorem}[section]
\newtheorem{prop}[theorem]{Proposition}
\newtheorem{lemma}[theorem]{Lemma}
\newtheorem{cor}[theorem]{Corollary}

\theoremstyle{definition}
\newtheorem{defn}[theorem]{Definition}
\newtheorem{rmk}[theorem]{Remark}

\newcommand\Z{{\mathbb{Z}}}
\newcommand\ZG{{\mathbb{Z}}G}
\newcommand\ra{{\rightarrow }}

\title[One-relator K\"ahler groups]{One-relator K\"ahler groups}

\author[I. Biswas]{Indranil Biswas}

\address{School of Mathematics, Tata Institute of Fundamental
Research, Homi Bhabha Road, Bombay 400005, India}

\email{indranil@math.tifr.res.in}

\author[M. Mj]{Mahan Mj}

\address{RKM Vivekananda University, Belur Math, WB-711202, India}

\email{mahan.mj@gmail.com; mahan@rkmvu.ac.in}

\subjclass[2000]{57M50,  32Q15, 57M05 (Primary);  14F35, 32J15 (Secondary)}

\date{\today}

\thanks{Research of the
second author is partly supported by  CEFIPRA Indo-French Research Grant 4301-1}

\begin{abstract}
We prove that a one-relator group $G$ is K\"ahler if and only if
either $G$ is finite cyclic or $G$ is isomorphic to
the fundamental group of a compact
orbifold Riemann surface of genus $g > 0$ with at most
 one cone point of order $n$:
$$
\langle a_1\, ,b_1\, , \,\cdots\, ,a_g\, ,b_g\, \mid\, 
(\prod_{i=1}^g [a_i\, ,b_i])^n\rangle\, .
$$
\end{abstract}

\maketitle

\section{Introduction}

Fundamental groups  of  compact K\"ahler manifolds, or
 {\it K\"ahler groups} for short, have attracted much attention 
(see \cite{abckt} for a survey of results and techniques). From a 
very different point of view, one-relator groups  have been 
studied for a long time in  combinatorial group theory 
\cite[Ch.~2]{ls}. (A one-relator group is the
quotient of a free group with finitely many generators by
one relation.) It is natural to ask which groups occur in the 
intersection of these two classes. In fact one-relator groups 
have appeared as test-cases for various restrictions developed 
for K\"ahler groups. Specific examples have been ruled out by
Arapura \cite[Section 7J]{arapura-exp}. Restrictions have been obtained from 
the point of view of rational homotopy theory (see
\cite[Sections~3,~4]{amoros-cmh} and \cite[p. 
39,~Examples~3.26,~3.27]{abckt}). Further 
restrictions follow from works of Gromov \cite{gromov-pi1} and 
Green and Lazarsfeld \cite{gl}.

In \cite{arapura-exp}, Arapura asks which one-relator groups are 
K\"ahler (see \cite[p. 12, Section J]{arapura-exp}). This
question was also raised by Amor{\'o}s. Our aim here is to
give a complete answer to Arapura's question. We prove the 
following (see Theorem \ref{main} and Section \ref{se-c}):

\begin{theorem}\label{th-i-a}
Let $G$ be an infinite one-relator group. Then $G$ is K\"ahler if 
and only if it is isomorphic to 
$$
\langle a_1\, ,b_1\, , \,\cdots\, ,a_g\, ,b_g\, \mid\, 
(\prod_{i=1}^g [a_i\, ,b_i])^n 
\rangle\, ,
$$
where $g$ and $n$ are some positive integers.
\end{theorem}

We show that each of the groups
$$
\langle a_1\, ,b_1\, , \,\cdots\, ,a_g\, 
,b_g\, \mid\, (\prod_{i=1}^g [a_i\, ,b_i])^n \rangle\, ,\, ~
g\, , n\, >\, 0,
$$
can in fact be realized as the fundamental group of a smooth 
complex projective variety.

It is known that every finite group is the fundamental group   
of some smooth complex projective variety (see \cite[p. 6, 
Example 1.11]{abckt}) and so all finite groups
are K\"ahler. Since the only finite one-relator groups are finite cyclic
groups, it follows that
finite one-relator K\"ahler groups are precisely the finite cyclic groups.

Therefore, Theorem \ref{th-i-a} has the following corollary:

\begin{cor}
Any torsion-free one-relator K\"ahler group is the fundamental
group of some closed orientable surface.
\end{cor}

We also prove the following closely related
result (see Corollary \ref{fxzcor2}):

\begin{theorem}
Let $G$ be a K\"ahler group such that
\begin{itemize}
\item it is a coherent group of rational cohomological 
dimension two, and

\item the virtual first Betti number of $G$ is positive.
\end{itemize}
Then $G$ is virtually a surface group.
\end{theorem}

We give an overview of the basic strategy of the proof:\\
It follows from the structure theory of one-relator groups that they are described as iterated HNN extensions.
The K\"ahler group $G$ we are interested in therefore acts on the Bass-Serre tree $T$ associated to the HNN splitting.
If $T$ is not quasi-isometric to the real line, it must be non-amenable and have infinitely many ends. It follows
from a refinement (Proposition \ref{bmst}) of the theory of stable cuts of Delzant-Gromov \cite{dg} using further 
structure of one-relator groups that $G$ is virtually a surface group in this case.

In case $T$ is quasi-isometric to the real line, then $G$ must be the mapping torus of a free group. These groups
are known to be coherent \cite{fh}. A simple cohomological dimension argument along with the structure of finitely
presented normal subgroups of cd 2 groups completes the proof in this case.

The torsion in $G$ is finally handled by further structure theory of one-relator groups.

\section{Preliminaries}

The reader is referred to Ch. VIII.10 of \cite{brown} for 
generalities on duality and Poincar\'e duality groups.

\begin{defn} \label{poincarebis}
A group $G$ is a {\it Poincar\'e duality group of dimension}
$n$ if $G$ is of type $FP$, and $$H^i(G, \ZG) \,= \,\begin{cases}
0, \quad \text{for} \; i \,\neq\, n\, , \\
\Z, \quad \text{for} \; i \,=\,n\, .
\end{cases}
$$

A group $G$ is a {\it duality group of dimension} $n$ if $G$
is of type $FP$, $$H^i(G, \ZG) \,=\, 0\, \, \quad \text{for} 
\; i \neq n\, ,$$ 
and $H^n(G, \, \ZG)$ is a non-zero torsion-free abelian group.
\end{defn}

Two facts about such groups that we will need are given below;
cohomological dimension is denoted by ``cd''.

\begin{theorem}[{\cite[Theorem 3.5]{be}}]\label{cdplus}
Let $1\,\longrightarrow\, N\, \longrightarrow\, G\,
\longrightarrow\, Q\, \longrightarrow\, 1$ be an exact sequence
of duality groups. Then $cd(G) = cd(Q) + cd(N)$.
\end{theorem}

We shall occasionally refer to a group of cohomological dimension $n$ 
as an $n$--dimensional group.

\begin{theorem}[{\cite[Theorem B]{bi}}]\label{cd2normal}
Let $G$ be a two-dimensional group, and let $H\, \subset \, G$
be a finitely presented normal subgroup of infinite index.
Then $H$ is free.
\end{theorem}

The reader is referred to \cite{dg} for the notion of cuts, 
particularly stable cuts, in K\"ahler groups. 

\begin{defn} Let $Q$ be a group and $R$ a subgroup of it with the 
following property. If $M$ is a Riemannian manifold with 
fundamental group $Q$ and $ M'$  is the cover of $M$ 
corresponding to $R$, then $M'$ has at least three ends and
is a non-amenable metric space.
We then say that $R$ is a {\it cut subgroup} of $Q$.
\end{defn}

A group $G$ is said to be {\it virtually a surface group} if some 
finite index subgroup of $G$ is the fundamental group of a closed
surface of positive first Betti number.

We shall need the following:

\begin{prop}[\cite{bms}]\label{bmst}
Let  $G$ be a  K\"ahler group fitting in a short  exact sequence  
$$1\,\longrightarrow\, N \,\longrightarrow\, G\,
\longrightarrow \,Q \,\longrightarrow\, 1\, ,$$
where $N$ is finitely generated, and one of the following holds:
\begin{itemize}
\item $Q$ admits a discrete faithful non-elementary minimal 
action on a simplicial tree with more than two ends;

\item $Q$ admits a (strong-stable) cut $R$ such that the 
intersection of all conjugates of $R$ is trivial.
\end{itemize} 
Then $G$ is virtually a surface group. 
\end{prop}

\section{One-relator Groups}

All one-relator groups will be assumed to be 
infinite and the generating set will be assumed to be finite
of cardinality greater than one. 
Our starting point is the following lemma due to Arapura.

\begin{lemma}[{\cite[Section 7J]{arapura-exp}}]\label{ara}
Suppose that $G \,=\, \langle x_1 \, , x_2 \, , \cdots \, , 
x_n\,\mid\, w(x_1, \cdots , x_n ) \rangle$ is a one-relator 
K\"ahler group with $n\, \geq \, 2$. Then the following statements 
hold:
\begin{enumerate}
\item $n$ is even.
\item Each $x_i$ occurs at least once in the word $w$, and the 
number of occurrences of $x_i$ and $x_i^{-1}$ coincide.
\item G surjects onto $\Gamma_g$ with $g = n/2$, where 
$\Gamma_g$ is the fundamental group of a closed
Riemann surface of genus $g$.
\end{enumerate}
\end{lemma}

\begin{rmk} Arapura proves Lemma \ref{ara} under the 
assumption that $n$ is strictly greater than two. However, if 
$n=2$, then it follows that the first Betti number of $G$ 
satisfies $1 \,\leq\, b_1(G) \,\leq\, 2$ and $b_1(G)\,=\,2$ if 
and only if the number of occurrences of $x_i$ and $x_i^{-1}$ in 
$w$ coincide for $i\,=\, 1,2$. Since $G$ is K\"ahler, 
we know that $b_1(G)$ is even, and hence 
$b_1(G)\, =\, 2$. Further, $G$ cannot split as 
a non-trivial free product \cite{gromov-pi1}. Hence each $x_i$ 
occurs at least once in the word $w$ for $i\,=\, 1,2$. Also, 
since $b_1(G)\,=\, 2$, the group $G$ surjects onto $\Gamma_1$.
\end{rmk}

We shall now recall some basic structure theory of one-relator groups due to Magnus and Moldavansky.
A subgroup $K$ of a one-relator group $G$ is called a {\it Magnus subgroup} if it is freely generated by a
subset of the generating set of $G$ that omits one of the 
generators present in the defining relator of $G$.
If $G$ is an HNN extension of $H$, and $t$ is the free letter 
conjugating subgroup $A \,\subset\, H$ to subgroup $B\,\subset\, 
H$ so that $$G\,=\,\langle H\, , t\,\mid\, A^t\,=\,B\rangle\, ,$$
then  $A$ and $B$ are called the {\it associated subgroups} of 
the HNN extension.

\begin{theorem}[{\cite{mo}, \cite[p. 250]{ms}}]\label{mm}
Let $G \,=\, \langle x_1 , \cdots , x_n \,\mid\, w \rangle$ be a 
one-relator group, where the exponent sum of $x_1$ in $w$ is 
zero. Then the following hold:
\begin{enumerate}
\item There exists a finitely generated
one-relator group $H$ whose defining relator has shorter length
than $w$.
\item $G$ is an HNN extension $G =H_{\ast F}$, where $F$ is a 
non-trivial free group.
\item The associated subgroups are Magnus subgroups of $H$.
\end{enumerate}
\end{theorem}

Finally we shall need two results due to Collins describing 
intersections of Magnus subgroups.

\begin{theorem}[{\cite[Theorem 1]{coll}}]\label{coll0} 
Let $G \,=\, 
\langle 
x_1 \, , \cdots \, , x_n \,\mid\, w \rangle$ be a one-relator 
group, 
where $w$ is cyclically reduced.
Let $M\,=\, F(S)$ and $N\,=\, F(T)$ be Magnus subgroups of $G$,
where $S$ and $T$ 
are subsets of the generating set, allowing $M\,= \,N$. If $M 
\bigcap N$ is distinct from $F(S \bigcap T)$, then $M \bigcap N$ 
is the free product of $F(S\bigcap T)$ and an infinite cyclic 
group.
\end{theorem}

\begin{theorem}[{\cite[Theorem 2]{coll}}]\label{coll}
Let $G$, $M$ and $N$ be as in Theorem \ref{coll0}. For any 
$g\,\in\, G$, either $gMg^{-1}\bigcap N$ is cyclic (possibly 
trivial) or $g\,\in\, NM$.
\end{theorem}

\begin{cor}\label{case1}
Let $G$, $S$, $M$ and $N$ be as in Theorem \ref{coll0}.
If $S$ has more than one element, and
$G \,\neq\, NM$, then there exist $$g_1\, , g_2\, , g_3 \,\in\, 
G$$ such that $g_1Mg_1^{-1} \bigcap g_2Mg_2^{-1} \bigcap N 
\bigcap g_3Ng_3^{-1}$ is trivial.
\end{cor}

\begin{proof}
By Theorem \ref{coll}, there exists $g\,=\,g_1\,\in\, G$ such 
that $g_1Mg_1^{-1}  \cap N \,\subset\, N$ is cyclic.
Let $g_1Mg_1^{-1}\bigcap N \,=\, \langle h \rangle$. 
Since $S$ has more than one element, it follows that $T$ also has 
more than one element. Hence $N$ is a free group on more than one 
generator. In particular, there exists $g_3 \in N$ such that $g_3 
\langle h \rangle g_3^{-1} \bigcap \langle h \rangle\,=\, \{1 
\}$. The proof is completed by setting $g_2 \,=\, g_3g_1$.
\end{proof}

A subgroup $H$ of a group $G$ is said to have {\it height} at 
most $n-1$ if for any $n$ elements $g_1\, , \cdots \, , g_n$ 
satisfying $Hg_i \,\neq\, Hg_j$ whenever $i\,\neq\, j$, the 
intersection $\bigcap_{i=1}^n g_iHg_i^{-1}$ is trivial.

\begin{prop}\label{case2}
Let $G \,=\, \langle x_1\, ,\cdots\, , 
x_n \,\mid\, w\rangle$ be  a one-relator group, where $w$ is 
cyclically reduced. Let $M \,=\, F(S)$ and $N \,=\, F(T)$ be 
Magnus subgroups of $G$, where $M\,\neq\, N$. Then there exist
elements $\{g_i\}_{i=1}^k \, , \{h_i\}_{i=1}^k\, \in\, G^k$, for 
some $k\,>\, 1$, such that $(\bigcap_{i=1}^k g_iMg_i^{-1}) 
\bigcap (\bigcap_{i=1}^k h_iNh_i^{-1})$ is trivial.
\end{prop}

\begin{proof}  By Theorem \ref{coll0}, the intersection $M\bigcap 
N$ is a finitely generated subgroup of both $M$ and $N$. As
$M\,\neq\, N$, we may assume without loss of generality that 
$M\bigcap N$ is an infinite index subgroup of $M$. Since any 
finitely generated infinite index subgroup of a free group has 
finite height \cite{gmrs}, the proposition follows.  
\end{proof}

\begin{prop} \label{case3} 
Let $G$, $M$ and $N$ be as in Corollary \ref{case1}. Assume that 
the intersection of all conjugates of $M$ together with all 
conjugates of $N$ is non-trivial. Then either $M\,=\,N\,=\,G$, or 
$M=N$ is an infinite cyclic normal subgroup of $G$.
\end{prop}

\begin{proof} By Proposition \ref{case2}, we have $M\,=\, N$. 
Assume that $G \,\neq\, M$. Then $M = N$ is infinite cyclic by 
Corollary \ref{case1}. If $M$ is not normal, then there exists 
$g\,\in\, G$ such that $gMg^{-1} \bigcap M$ is a proper subgroup 
of $M$. Hence $\bigcap_{i=0}^\infty g^iMg^{-i}$ must be trivial. 
This completes the proof of the proposition.
\end{proof}

We shall need the following Theorem due to Karrass and Solitar in 
the proof of Theorem \ref{1rel} below.

\begin{theorem}[{\cite[p. 219]{ks}}]\label{ks} Let $G$ be a 
one-relator group having a (nontrivial) finitely presented normal 
subgroup $H$ of infinite index. Then $G$ is torsion-free and has 
two generators. Further, $G$ is an infinite cyclic or infinite dihedral 
extension of a finitely generated free group $N (\subset G)$ satisfying the following: 
\begin{itemize}
\item $H \,\subset\, N$ if $H$ is not
cyclic, and

\item $H\bigcap N$ is trivial if $H$ is cyclic.
\end{itemize}
In either case, there is a finite index subgroup $G_1$
of $G$ and a finitely generated free group $F$ fitting in an
exact sequence
$$
1 \longrightarrow F \longrightarrow  G_1\longrightarrow  
\Z\longrightarrow 1\, .
$$
\end{theorem}

\begin{theorem} \label{1rel}
Let $G \,=\, \langle x_1\, ,\cdots \, , x_n\,\mid\, w  \rangle$ 
be a one-relator K\"ahler group. Then either $G$ is virtually a 
surface group, or there is a finite index subgroup $G_1$
of $G$ and  a finitely generated free group $F$ fitting in an
exact sequence
$$
1 \longrightarrow F \longrightarrow  G_1\longrightarrow
\Z\longrightarrow 1\, . 
$$
\end{theorem}

\begin{proof}
By Lemma \ref{ara}, the exponent sum of $x_1$ in $w$ is zero. 
Hence from Theorem \ref{mm} it follows that there exists a 
finitely generated one-relator group $H$ such that $G$ is an HNN 
extension $G\,=\, H_{\ast F}$, where  $F$ is a non-trivial free 
group. Also, the associated subgroups $M_1$ and $M_2$ are Magnus 
subgroups of $H$.

Let $T$ be the Bass--Serre tree of the splitting of $G$ over $H$. 
If $T$ is quasi-isometric to $\mathbb{R}$, then there is a finite 
index subgroup $G_1$ of $G$ and  a finitely generated free group 
$F$ fitting in an exact sequence
$$
1 \longrightarrow F \longrightarrow  G_1\longrightarrow
\Z\longrightarrow 1\, ,
$$
where $F$ is of finite index in both $M_1$ and $M_2$; therefore 
$H$ is free. 

Assume that $T$ is not quasi-isometric to $\mathbb{R}$. Then $H$ 
is a stable cut subgroup in the sense of 
\cite{dg}. This is because $T$ must have infinitely many ends and hence be non-amenable.
Hence by \cite{dg} (or by Proposition \ref{bmst}), 
there is a surjective homomorphism from $G$ to a surface group 
$\Gamma_g$ of genus $g > 0$ with  kernel $N\, \subset
\, H$. If $N$ is finite, then $G$ is virtually a surface group.

Assume that $N$ is infinite. Let $t$ be the stable letter of the 
HNN extension $G\,=\, H_{\ast F}$. Hence $N \subset H \bigcap 
tHt^{-1} \,\subset\, M_1 \bigcap M_2$.
Since $N$ is non-trivial, from Proposition \ref{case3} we 
conclude that
\begin{enumerate}
\item[(a)] either $M_1\,=\,M_2\,=\,H$,
\item[(b)] or $M_1\,=\,M_2$ is an infinite cyclic normal 
subgroup of $H$.
\end{enumerate}

In Case (a), the group $G$ fits in $1\longrightarrow F 
\longrightarrow G\longrightarrow \Z\longrightarrow 1$, where
$F\, =\,M_1\,=\,M_2\,=\, H$ is free.

In Case (b), the subgroup $N\,\subset\, M_1 \bigcap M_2$ must be 
infinite cyclic, in particular, $G$ admits an infinite cyclic 
normal subgroup. Hence by Theorem \ref{ks}, there is a finite 
index subgroup $G_1$ of $G$ and a finitely generated free group 
$F$ fitting in an exact sequence $1\longrightarrow 
F\longrightarrow G_1\longrightarrow \Z\longrightarrow 1$. 
This proves the theorem.
\end{proof}

In the next section we shall rule out the second possibility in
Theorem \ref{1rel}.

\section{Coherent Groups}

\subsection{Torsion-free one-relator K\"ahler groups}

A finitely presented group $G$ is said to be {\it coherent} if 
every finitely generated subgroup of $G$ is
finitely presented. We shall be requiring the following deep 
theorem due to Feighn and Handel:

\begin{theorem}[\cite{fh}]\label{fh}
Let $G$ be a finitely presented group admitting a description
$$1\longrightarrow F \longrightarrow G\longrightarrow 
\Z\longrightarrow 1\, ,$$
where $F$ is finitely generated free. Then $G$ is coherent.
\end{theorem}

A group is said to 
be {\em curve-dominating} if it possesses a surjective 
homomorphism onto the fundamental group $\Gamma_g$ of a closed 
Riemann surface of genus $g\,>\, 1$. 

\begin{theorem}[{\cite[Sections 7J, K]{arapura-exp}}]\label{ara-curve} 
Let $G$ be a curve-dominating K\"ahler group.
Then $$\dim H_1 (G^\prime , \, \mathbb{Q} ) \,=\, \infty\, ,$$ 
where $G^\prime$ denotes the commutator subgroup of $G$.
Conversely, if $\dim H_1 (G^\prime ,\, \mathbb{Q} )\,=\, \infty$, 
then $G$ contains a curve-dominating
subgroup of finite index. \end{theorem}

\begin{prop}\label{fxz}
Let $G$ be a finitely presented group admitting a description
$$1\longrightarrow F \longrightarrow G\longrightarrow
\Z\longrightarrow 1\, ,$$ where $F$ 
is  finitely generated free. If $G$ is K\"ahler, then $G\,=\,\Z \oplus \Z$.
\end{prop}

\begin{proof}
Suppose that $G\,\not=\,\Z \oplus \Z$. Then $F$ has 
rank greater than one; the only other case is that of the 
fundamental group of a Klein bottle, which has first Betti number 
one and hence it cannot be K\"ahler. As $G/F$ is abelian, it 
follows that $G^\prime\,\subset\, F$ (as before, $G^\prime$ is 
the commutator subgroup of $G$). Hence $G^\prime$ is a 
normal subgroup of $F$. Further, since $G$ is non-abelian (as $F$ 
has rank greater than one), we have
$G^\prime\,\neq\, \{ 1 \}$. As subgroups of free groups are 
free, it follows that $G^\prime$ is a free group.

Since $b_1(G)$ is even and $b_1(G)\geq 1$, it follows that  
$G^\prime \,\subset\, F$ must be of infinite index in $F$. Hence  
$\dim H_1 (G^\prime ,\, \mathbb{Q} ) \,=\, \infty$. So
$G$ contains a curve-dominating subgroup of finite index by 
Theorem \ref{ara-curve}.

By the Siu--Beauville Theorem \cite[p. 2, Theorem 1.5]{abckt}, 
there exists a 
holomorphic map $$\phi\,:\, X \,\longrightarrow\, S\, ,$$ from a 
compact K\"ahler manifold to a closed Riemann surface $S$ of 
genus $g\,>\,1$, with connected fibers. 

It now follows (cf. \cite[p. 283, Lemma 3]{cat-fib}) \footnote{We are
grateful to Dieter Kotschick for 
 informing  us that we need this refinement to get rid of multiple fibers 
and for our argument to work.} that there exists a finite-sheeted
cover $X_1$ of $X$ and  a closed Riemann surface $S_1$  (possibly different from $S$) of 
genus greater than one such that $\phi$  lifts to a
holomorphic map $$\phi_1\,:\, X_1 \,\longrightarrow\, S_1\, ,$$ from $X_1$
 to  $S_1$, with connected fibers and no multiple fibers.

 As the
fibers of $\phi$ are connected and compact and there are no multiple fibers, it follows that the 
kernel $N$ of $\phi_{1\ast}$ is finitely generated. By Theorem 
\ref{fh}, this group $N$ is finitely presented. Since $N\, 
\subset\, G_1$ 
has infinite index, it must be  finitely generated free by Theorem \ref{cd2normal}. 
Thus $G_1$ fits in an exact sequence $$1\,\longrightarrow\, F 
\,\longrightarrow\, 
G_1\,\longrightarrow\, \pi_1(S_1) \,\longrightarrow\, 1\, ,$$ 
where $F$ is finitely generated free and hence a one dimensional 
duality group. From Theorem \ref{cdplus} it follows that $G_1$ has 
cohomological dimension $1+2=3$. Hence (by Serre's theorem on finite index subgroups) $G$ has 
cohomological dimension $3$.

Again, since $G$ fits in an exact sequence $$1\,\longrightarrow\, 
F \,\longrightarrow\,
G\,\longrightarrow\, \Z \,\longrightarrow\, 1\, ,$$
where $F$ is  finitely generated free, it follows from Theorem \ref{cdplus} again
that $G$ has cohomological dimension $1+1=2$. This is in 
contradiction with the previous calculation.

Hence $F$ has rank one, and $G\,=\,\Z \oplus \Z$.
\end{proof}

Combining Theorem \ref{1rel} and Proposition \ref{fxz} we have:

\begin{theorem}\label{1relf}
Let $G \,=\, \langle x_1\, ,\cdots \, , x_n\,\mid\, w  \rangle$ 
be a one-relator K\"ahler group. Then $G$ is virtually a surface 
group. Further, if $G$ is torsion-free, then it is isomorphic to 
a surface group.
\end{theorem}

\begin{proof}
The first statement in the theorem is a direct consequence of 
Theorem \ref{1rel} and Proposition \ref{fxz} because $\Z \oplus 
\Z$ is a surface group. The last statement follows from the fact 
that a torsion-free group that is virtually a surface group is 
actually a surface group.
\end{proof}

The case where $G$ has torsion will be dealt in Section 
\ref{sec-t}.

The proof of Proposition \ref{fxz} actually gives the following:

\begin{cor}\label{fxzcor}
Let $G_0$ be a coherent K\"ahler group of rational cohomological
dimension two such that a finite index subgroup $G$ of $G_0$ 
admits a surjective homomorphism $$\phi_\ast\,: \,G \,
\longrightarrow\, \pi_1(S)\, ,$$ where $S$ is a closed oriented 
2-dimensional surface of genus $g\,>\,1$. Then $G_0$ is virtually 
a (real two-dimensional) surface group.
\end{cor}

\begin{proof}
As in the proof of Proposition \ref{fxz}, there exists a 
holomorphic map $\psi\,:\, X \,\longrightarrow\, T$ from a 
compact K\"ahler 
manifold to a closed Riemann surface $T$ of genus $g>1$ with 
connected fibers inducing a surjection $\psi_\ast$ of $G$ onto 
$\pi_1(T)$; we note that $T$ need not be the same as $S$ but 
will, in general, be a finite-sheeted cover of $S$. Again, the 
kernel $N$ of $\psi_\ast$ is finitely presented. If $N$ is 
infinite, then the rational cohomological dimension must be 3 by 
Theorem \ref{cdplus}, contradicting the hypothesis. Hence $N$ 
must be finite, and the result follows.
\end{proof}

The virtual first Betti number of a manifold 
$M$ is defined to be
$$vb_1(M) \, := \, {\rm Sup}_N \{ b_1(N) \,\mid\, {\rm N \, ~ is 
~ a 
~  \text{finite--sheeted} ~ cover ~ of ~ M }\}\, .$$
Similarly, for a group $G$, define the virtual first Betti number 
$vb_1(G)$ to be
$$vb_1(G) \,:=\, sup_H \{ b_1(H) \,\mid\, {\rm H \, ~ is ~ a ~ 
\text{finite--index} ~ subgroup ~ of ~ G }\}\, .$$

To generalize  Theorem \ref{fxz} further, we shall require 
certain properties of the Albanese map. 
The following Lemma will be used to strengthen Proposition \ref{fxz}
to Corollary  \ref{fxzcor2} below. This is a special case of a Theorem of
Catanese, \cite[p. 23, Proposition 2.4]{abckt},
and a simple self-contained proof may be found in
\cite[Section 2]{kotschick}
by taking $M$ (in the statement of Lemma \ref{alb}
below) to be a cover with positive $b_1$.

\begin{lemma} \label{alb}
Let $M_0$ be a compact K\"ahler manifold, with $\pi_1(M_0)\,=\, 
G$,
such that the real cohomological dimension $cd_{\mathbb{R}} (G) 
\, <\, 4$, and $vb_1(M_0)\, >\, 0$.  Then there exists a 
finite-sheeted cover $M$ of $M_0$ such that the image of the Albanese map $F$ for $M$ is a smooth algebraic curve $S$
of genus greater than zero. 
\end{lemma}

Combining Lemma \ref{alb} with the proof of Proposition \ref{fxz} we get the following.

\begin{cor}\label{fxzcor2}
Let $G$ be a coherent K\"ahler group of rational cohomological
dimension two  and $vb_1(G)\, >\, 0$. Then $G$ is virtually 
a (real two-dimensional) surface group.
\end{cor}

\begin{proof} By Lemma  \ref{alb}, we obtain a holomorphic map
$F : M \longrightarrow S$ onto a 
Riemann surface $S$
of genus greater than zero, where $\pi_1(M)$ is a finite index subgroup of $G$. 
By the Stein factorization theorem, we can assume (by passing to a further finite-sheeted cover if necessary) that the fibers of $F$ are connected.
As in the proof of Proposition \ref{fxz}, we can further assume that $F$ has no multiple fibers.
Hence we have an exact sequence
$$1 \longrightarrow N \longrightarrow \pi_1(M) \longrightarrow \pi_1(S) \longrightarrow 1$$
with $N$ finitely generated, forcing $N$ to be finitely generated free as in  Proposition \ref{fxz}. This
 forces the rational cohomological dimension of $\pi_1(M)$ to be three, contradicting the hypothesis.
\end{proof}

\section{Orbifold Groups}

\subsection{One-relator groups with torsion}\label{sec-t}

Throughout this subsection, $G\,=\, \langle x_1\, ,\cdots \, 
,x_k\, \mid\, w^n\rangle$ will  be a one-relator group, where 
$k\,>\,1$, $n\, \geq\, 1$, and $w$ is cyclically reduced and not 
a proper power.

If $n\,=\, 1$, then it is known that $G$ is torsion-free
\cite[p. 266, Theorem 4.12]{MKS}. Therefore, in view
of Theorem \ref{1relf}, we are allowed to assume that 
$n\, >\, 1$.

Fischer, Karrass and Solitar  prove the following:

\begin{prop}[{\cite[Theorem 1]{fks}}]\label{fks} Let
$G\,=\, \langle x_1\, ,\cdots \, ,x_k\, \mid\, w^n\rangle$
be a one-relator group with $n \, >\, 1$. Then the following
statements hold:
\begin{enumerate}
\item every torsion element in $G$ is conjugate to a power of 
$w$, and

\item the subgroup generated by torsion elements in $G$
is the free product of the conjugates of $w$.
\end{enumerate}
\end{prop}

Murasugi has described in detail the centers of one-relator groups.

\begin{theorem}[{\cite[Theorems 1, 2]{mu}}]\label{center} Let 
$G\,=\, \langle x_1\, ,\cdots \, , x_k\, \mid\, w^n\rangle$ be a 
one-relator group with $n\geq1$. If $k\geq 3$, then the center 
$Z(G)$ of $G$ is trivial. If $G$ is non-abelian, $k=2$ and  
$Z(G)$  is non-trivial, then $Z(G)$ is infinite cyclic.
\end{theorem} 

\begin{prop} Let
$$
G\,=\, \langle x_1\, ,\cdots \, ,x_k\, \mid\, w^n\rangle\, ,
~\, n\, >\, 1\, , 
$$
be a one-relator K\"ahler group that contains a finite index 
subgroup isomorphic to the fundamental group of a closed 
orientable surface $\Sigma$ of genus greater than one. Then $G$ 
is isomorphic to the fundamental group of a hyperbolic (real) 
two-dimensional compact orbifold $V$ with exactly one cone-point. 
Further, the underlying manifold of $V$ is orientable.
\label{orb1} \end{prop}

\begin{proof}
The group $G$ is clearly a (Gromov) hyperbolic with 
boundary homeomorphic to the circle $S^1$
\cite{gromov-hypgps}. The group $G$ acts naturally on the 
boundary $S^1$. Let $N$ be the kernel of the action, meaning 
$$N\, =\, \{z\,\in\, G\, \mid\, z(x)\,=\, x ~\, ~\, ~\forall ~ 
x \,\in\,  S^1 \}\, .$$
Then $G$ is isomorphic to $N \times (G/N)$, where $G/N$ acts 
effectively on $S^1$ as a convergence group. By a deep theorem of 
Casson--Jungreis \cite{casson-j}, and (independently) Gabai 
\cite{gabai-cgnce}, it follows that $G/N$ is the fundamental 
group of a compact hyperbolic orbifold $V$ of dimension two.

By Proposition \ref{fks}, the kernel $N$ must be cyclic. Hence 
$N$ is contained in the center of $G$. Therefore, by Theorem 
\ref{center}, the group $N$ is trivial.  Finally, since all 
torsion elements of $G$ are conjugate, the orbifold $V$ must have 
a unique cone point. In fact, by the explicit description of 
presentations of orbifold groups given in \cite{scott-geoms},
\begin{enumerate}
\item[(a)] either $G \,=\, \langle a_1\, ,b_1\, , 
\cdots\, ,a_g\, ,b_g\, \mid\, w^n \rangle$, 
where $w= \prod_{i=}^g [a_i\, ,b_i]$, and the underlying manifold 
is orientable,

\item[(b)] or $G \,=\, \langle a_1\, , \cdots\, ,a_g\, \mid\, w^n 
\rangle$, where $w= \prod_{i=}^g a_i^2$, and the underlying 
manifold is non-orientable.
\end{enumerate}

Case (b) is ruled out by Lemma \ref{ara}. This completes the 
proof of the proposition.
\end{proof}

\begin{prop}\label{orb2}
Let $G\,=\, \langle x_1\, ,\cdots \, ,x_k\, \mid\, w^n\rangle$  
be a one-relator K\"ahler group such that $G$ contains a finite index 
subgroup isomorphic to $\Z \oplus \Z$. Then $G$ is isomorphic
to $\Z \oplus \Z$.
\end{prop}

\begin{proof} By Lemma \ref{ara}, the integer $k$ is even and $G$ 
admits a surjection onto the fundamental group of a closed 
orientable surface $S$ of genus $k/2$. Since $G$ contains a 
finite index subgroup isomorphic to $\Z \oplus \Z$, it follows 
that $\Z \oplus \Z$ admits a surjection onto the fundamental 
group of a finite sheeted cover of $S$. Hence $k\,=\,2$. 
Consequently, $G$ admits a right-split exact sequence
$$
1 \,\longrightarrow\, N \,\longrightarrow\, G\,\longrightarrow
\,Q ~\, (=\, \Z \oplus \Z)\,\longrightarrow\, 1\, .
$$ 
Let $H \subset G$ be a subgroup mapping isomorphically onto $Q$.

As in the proof of Proposition \ref{orb1}, the group $N$ must be
finite cyclic by Proposition \ref{fks}. It follows that there 
exists a positive integer $n$ such that $h^n$ is in 
the center of $G$ for each $h\,\in\, H$. In particular, the 
center of $G$ contains $\Z \oplus \Z$.
By Theorem \ref{center}, the group $G$ must be abelian. 

By Lemma \ref{ara} again, the abelianization of $G$ is exactly 
$\Z \oplus \Z$. Since $G$ is itself abelian, the proposition 
follows.
\end{proof}

\subsection{The main theorem}

Combining Theorem \ref{1relf},  Proposition 
\ref{orb1} and Proposition \ref{orb2}, we  have:

\begin{theorem}\label{main}
Let $G\,=\, \langle x_1\, ,\cdots \, ,x_k\, \mid\, w\rangle$ be 
an infinite one-relator K\"ahler group. Then
$$
G \,=\, \langle a_1\, ,b_1\, , \cdots\, ,a_g\, ,b_g\, \mid\, 
(\prod_{i=1}^g [a_i,b_i])^n \rangle
$$
for some $g\,\geq\, 1$ and $n\,\geq\, 1$.
\end{theorem}

\subsection{Examples}\label{se-c}

Take positive integers $g$ and $n$. We will show that the group
$$
G \,:=\, \langle a_1\, ,b_1\, , \cdots\, 
,a_g\, ,b_g\, \mid\, (\prod_{i=1}^g [a_i,b_i])^n \rangle
$$
is the fundamental group of a smooth complex projective variety.
We assume that $n\, >\, 1$, because $G$ for $n\,=\, 1$ is
a surface group.

Let $f\, :\, X\, \longrightarrow\, S$ be a smooth projective
elliptic surface satisfying the following conditions:
\begin{itemize}
\item the genus of the Riemann surface $S$ is $g$,

\item there is a point $x_0\, \in\, S$ such that the gcd of the
multiplicities of the irreducible components of the fiber over 
$x_0$ is $n$,

\item the reduced fiber over some point (it can be $x_0$) is 
singular, and

\item for each point $x\, \in\, S\setminus \{x_0\}$, the
gcd of the multiplicities of the irreducible components of the 
fiber over $x$ is $1$.
\end{itemize}

The fundamental group of $X$ is the group $G$ defined above.
To see this, consider the short exact sequence in the bottom half 
of \cite[p. 600]{xiao}. The group ${\mathcal H}_f$ in this exact 
sequence coincides with $G$ \cite[p. 601, Lemma 2]{xiao}. The 
group ${\mathcal V}_f$ in the exact sequence vanishes \cite[p. 
614, Theorem 4]{xiao}. Therefore, from this short exact sequence 
we conclude that $\pi_1(X) \,=\, G$.

\medskip
\noindent
{\bf Acknowledgments:} Different parts of this work were done 
while the authors were vising Harish--Chandra Research Institute, 
Allahabad, and Institute of Mathematical Sciences, Chennai and 
during a visit of the second author to Tata Institute of 
Fundamental Research, Mumbai. We thank these institutions for 
their hospitality. We would like to thank Dieter Kotschick for pointing out a small gap in an earlier version
of Proposition \ref{fxz} and for pointing out the proof in Section 2 of \cite{kotschick} to
us. Recently Kotschick \cite{kotschick2} has found a new proof of Theorem \ref{main} by techniques involving
$l^2$ cohomology.

\end{document}